\documentclass[leqno]{amsart}
\usepackage{amssymb}
\usepackage{mathrsfs}
\usepackage{amsmath, amsfonts, vmargin, enumerate}
\usepackage{color}
\usepackage{verbatim}
\usepackage{amsthm}
\usepackage{latexsym, bm}

\usepackage{euscript}
\usepackage{dsfont}

\input xypic

\xyoption {all}

 \makeatletter

\newtheorem{thm}{Theorem}

\newtheorem{lem}[thm]{Lemma}
\newtheorem{defi}[thm]{Definition}
\newtheorem{remark}[thm]{Remark}
\newtheorem{example}[thm]{Example}
\newtheorem{pb}[thm]{Problem}

\newenvironment{rk}{\begin{remark}\rm}{\end{remark}}

\newcommand{\B}{{\mathcal B}}

\newcommand{\M}{{\mathcal M}}

\newcommand{\R}{{\mathcal R}}

\renewcommand{\a}{\alpha}

\renewcommand{\t}{\theta}
\newcommand{\e}{\varepsilon}

\renewcommand{\O}{\Omega}

\newcommand{\ot}{\otimes}

\newcommand{\8}{\infty}
\newcommand{\el}{\ell}

\newcommand{\n}{\noindent}

\newcommand{\be}{\begin{align*}}
\newcommand{\ee}{\end{align*}}
\newcommand{\beq}{\begin{equation}}
\newcommand{\eeq}{\end{equation}}
\newcommand{\beqn}{\begin{equation*}}
\newcommand{\eeqn}{\end{equation*}}

\begin{document}

\title[Real interpolation of operator $L_p$-spaces]{Notes on real interpolation of operator $L_p$-spaces}

\thanks{{\it 2000 Mathematics Subject Classification:} Primary:  Secondary: }
\thanks{{\it Key words:} Operator spaces, $L_p$-spaces, real interpolation, column Hilbertian spaces}

 \author[Marius Junge]{Marius Junge}
 \address{Department of Mathematics,  University of Illinois, Urbana, IL 61801, USA}
 \email{junge@math.uiuc.edu}

 \author[Quanhua  Xu]{Quanhua Xu}
\address{Institute for Advanced Study in Mathematics, Harbin Institute of Technology,  Harbin 150001, China; and
Laboratoire de Math{\'e}matiques, Universit{\'e} de Bourgogne Franche-Comt{\'e}, 25030 Besan\c{c}on Cedex, France}
\email{qxu@univ-fcomte.fr}

\date{}
\maketitle

 \begin{abstract}
 Let $\mathcal{M}$ be a semifinite von Neumann algebra. We equip the associated noncommutative $L_p$-spaces with their natural operator space structure introduced by Pisier via complex interpolation. On the other hand, for $1<p<\infty$ let
  $$L_{p,p}(\mathcal{M})=\big(L_{\infty}(\mathcal{M}),\,L_{1}(\mathcal{M})\big)_{\frac1p,\,p}$$
 be equipped with the operator space structure via real interpolation as defined by the second named author ({\em J. Funct. Anal}. 139 (1996), 500--539). We show that $L_{p,p}(\mathcal{M})=L_{p}(\mathcal{M})$ completely isomorphically if and only if $\mathcal{M}$ is finite dimensional. This solves in the negative the three problems left open in the quoted work of the second author.

We also show that for $1<p<\infty$ and $1\le q\le\infty$ with $p\neq q$
  $$\big(L_{\infty}(\mathcal{M};\ell_q),\,L_{1}(\mathcal{M};\ell_q)\big)_{\frac1p,\,p}=L_p(\mathcal{M}; \ell_q)$$
with equivalent norms, i.e., at the Banach space level if and only if $\mathcal{M}$ is isomorphic, as a Banach space, to a commutative von Neumann algebra.

Our third result concerns the following inequality:
 $$
 \big\|\big(\sum_ix_i^q\big)^{\frac1q}\big\|_{L_p(\mathcal{M})}\le\big\|\big(\sum_ix_i^r\big)^{\frac1r}\big\|_{L_p(\mathcal{M})}
 $$
 for any finite sequence $(x_i)\subset L_p^+(\mathcal{M})$, where $0<r<q<\infty$ and $0<p\le\infty$. If $\mathcal{M}$ is not  isomorphic, as a Banach space, to a commutative von Meumann algebra, then this inequality holds if and only if $p\ge r$.
 \end{abstract}

\bigskip


\section{Introduction}


Interpolation of $L_p$-spaces is a classical subject. Our reference for interpolation theory is \cite{BL}. Let $(\O, \mu)$ be a measure space. Let $1\le p_0,p_1, p\le\8$ and $0<\t<1$ with $\frac1p=\frac{1-\t}{p_0}+ \frac{\t}{p_1}$. The following well known interpolation equalities hold
 \beq\label{comLp}
 \big(L_{p_0}(\O),\, L_{p_1}(\O)\big)_{\t}=L_{p}(\O)\;\text{ with equal norms},
 \eeq
 \beq\label{realLp}
 \big(L_{p_0}(\O),\, L_{p_1}(\O)\big)_{\t, p}=L_{p}(\O) \;\text{ with equivalent norms}.
 \eeq
Here $(\,\cdot,\,\cdot\,)_{\t}$ and  $(\,\cdot,\,\cdot\,)_{\t, p}$ denote, respectively,  the complex and real interpolation functors.  It is also well known that the above equalities hold for vector-valued $L_p$-spaces. More precisely, under the same assumption on the parameters  (assuming additionally $p<\8$), we have
 \beq\label{comVLp}
 \big(L_{p_0}(\O; E_0),\, L_{p_1}(\O;E_1)\big)_{\t}=L_{p}\big(\O; (E_0,\, E_1)_{\t}\big) \;\text{ with equal norms},
 \eeq
  \beq\label{realVLp}
  \big(L_{p_0}(\O; E_0),\, L_{p_1}(\O;E_1)\big)_{\t,p}=L_{p}\big(\O; (E_0,\, E_1)_{\t,p}\big) \;\text{ with equivalent norms}
 \eeq
for any compatible pair $(E_0,\,E_1)$ of Banach spaces.

\smallskip

The present note concerns interpolation theory in the category of operator spaces. We refer to \cite{ER, pis-intro} for operator space theory. The complex and real interpolations for operator spaces are developed in \cite{pis-OH} and \cite{xu-op}, respectively. Unless explicitly stated otherwise, all (commutative and noncommutative) $L_p$-spaces in the sequel are equipped with their natural operator space structure as defined in \cite{pis-NCLp, pis-intro}.  Pisier proved that \eqref{comLp} and \eqref{comVLp} remain true in the category of operator spaces, that is, these equalities hold completely isometrically,  $(E_0,\,E_1)$ being, of course, assumed to be operator spaces in the case of \eqref{comVLp}.

It is more natural to work with noncommutative $L_p$-spaces in the category of operator spaces. Let $\M$ be a semifinite von Neumann algebra equipped with a normal semifinite faithful trace $\tau$. Let $L_p(\M)$ denote the associated noncommutative $L_p$-space (cf. \cite{px}). If $\M=B(\el_2)$ with the usual trace, $L_p(\M)$ is the Schatten $p$-class $S_p$. If $\M$ is hyperfinite, Pisier \cite{pis-NCLp} also introduced the vector-valued $L_p(\M;E)$ for an operator space $E$, and showed that \eqref{comVLp} continues to hold in this more general setting:
  \beq\label{comVNLp}
 \big(L_{p_0}(\M; E_0),\, L_{p_1}(\M;E_1)\big)_{\t}=L_{p}\big(\M; (E_0,\, E_1)_{\t}\big) \;\text{ completely isometrically}
 \eeq
 for any hyperfinite $\M$ and any compatible pair $(E_0,\, E_1)$ of operator spaces.

However, real interpolation does not behave as smoothly as complex interpolation in the category of operator spaces. The problem whether \eqref{comVNLp} can be extended to real interpolation was left unsolved in \cite{xu-op},  see Problems~6.1, 6.2 and 6.4 there.  Let $1<p<\8$. Using the Banach space equality
  $$L_{p}(\M)=\big(L_{\8}(\M),\, L_{1}(\M)\big)_{\frac1p, p},$$
 we can equip $L_{p}(\M)$ with another operator space structure via real interpolation as in \cite{xu-op}, the resulting operator space is denoted by $L_{p, p}(\M)$. Then Problem~6.1 of \cite{xu-op} asks whether $L_{p, p}(\M)=L_{p}(\M)$ completely isomorphically for $p\neq2$ (the answer is affirmative for $p=2$). The following result resolves this problem in the negative.

 \begin{thm}\label{Lpp}
  Let  $1<p<\8$ with $p\neq2$. Then $L_{p, p}(\M)=L_{p}(\M)$ completely isomorphically if and only if $\M$ is finite dimensional.
   \end{thm}

Consequently, the answers to all three Problems~6.1, 6.2 and 6.4 of \cite{xu-op} are negative. In particular, neither \eqref{realLp} nor \eqref{realVLp} holds in the category of operator spaces.

\medskip

The next theorem provides an even worse answer to Problem~6.4. It shows that \eqref{comVNLp} extends to real interpolation at the Banach space level only in the commutative case. Recall that one can define $L_p(\M;\el_q)$ for any von Neumann algebra $\M$ (see section~\ref{Proof of Theorem2} for more information). $L_p(\M;\el_q)$ coincides with Pisier's space when $\M$ is hyperfinite. Note that $L_p(\M;\el_q)$ is defined only as a Banach space if $\M$ is not hyperfinite.

 \begin{thm}\label{Lppb}
  Let  $1<p<\8$ and $1\le q\le\8$ with $p\neq q$. Then
  \beq\label{Lpinfinite}
  \big(L_{\8}(\M; \el_q),\, L_{1}(\M;\el_q)\big)_{\frac1p,p}=L_{p}(\M; \el_q) \;\text{ with equivalent norms}
  \eeq
   if and only if $\M$ is isomorphic, as a Banach space, to $L_\8(\O,\mu)$ for some measure space $(\O,\mu)$.
\end{thm}

\medskip

Our third theorem does not concern the real interpolation of the $L_{p}(\M; \el_q)$ spaces but gives a result that is to be compared with the norm of these spaces. In the commutative case, the norm of a sequence $(x_i)$ in  $L_{p}(\O; \el_q)$ is given by
 $$\|(x_i)\|_{L_{p}(\O; \el_q)}=\big\|\big(\sum_i|x_i|^q\big)^{\frac1q}\big\|_{L_p(\O)}\,.$$
It is well known that this expression is no longer valid in the noncommutative setting, which is one source of many difficulties in noncommutative analysis. The following theorem shows that another classical property of the norm $\|(x_i)\|_{L_{p}(\O; \el_q)}$ does not extend to the noncommutative case. The index $p$ is now allowed to go below $1$, $L_p^+(\M)$ denotes the positive cone of $L_p(\M)$.

\begin{thm}\label{decreasing}
 Let $0<r<q<\8$ and  $0<p\le\8$.
 \begin{enumerate}[\rm(i)]
 \item If $p\ge r$, then
 \beq\label{decreasing in q}
 \big\|\big(\sum_ix_i^q\big)^{\frac1q}\big\|_{L_p(\M)}\le\big\|\big(\sum_ix_i^r\big)^{\frac1r}\big\|_{L_p(\M)}
 \eeq
 for any finite sequence $(x_i)\subset L_p^+(\M)$.
\item If $p<r$ and $\M$ is not  isomorphic, as a Banach space, to $L_\8(\O,\mu)$ for some measure space $(\O,\mu)$, then there exists no constant
$C$ such that
  \beq\label{negative}
 \big\|\big(\sum_ix_i^q\big)^{\frac1q}\big\|_{L_p(\M)}
 \le C\,\big\|\big(\sum_ix_i^r\big)^{\frac1r}\big\|_{L_p(\M)}
 \eeq
for any finite sequence $(x_i)\subset L_p^+(\M)$.
 \end{enumerate}
\end{thm}

The previous theorems will be respectively proved in the next three sections.


\section{Proof of Theorem~\ref{Lpp}}


We will need some preparations on column and row Hilbertian operator spaces. Let $C_p$ (resp.
$R_p$) denote the first column (resp. row) subspace of $S_p$ consisting of matrices whose
all entries but those in the first column (resp. row) vanish.  We have the following completely isometric
identifications:
 \beq\label{duality Cp-Rp}
 (C_p)^*\cong C_{p'}\cong R_p\quad\mbox{and}\quad
 (R_p)^*\cong R_{p'}\cong C_p\,,\quad\forall\;1\le p\le\8,
 \eeq
where $p'$ denotes the conjugate index of $p$. $C_p$ and $R_p$ can  be also defined via complex interpolation from $C=C_\8$
and $R=R_\8$. We view $(C, R)$ as a compatible pair by identifying
both of them with $\ell_2$ (at the Banach space level), i.e., by
identifying the canonical bases $(e_{k,1})$ of $C_p$ and
$(e_{1,k})$ of $R_p$ with $(e_k)$ of $\ell_2$. Then we have the following completely isometric equalities
\beq\label{CR}
C_p=(C,\ R)_{\frac1p}=(C_\8,\ C_1)_{\frac1p}\quad\mbox{and}\quad
 R_p=(R,\ C)_{\frac1p}=(R_\8,\ R_1)_{\frac1p}.
 \eeq
We refer to \cite{pis-OH, pis-NCLp} for more details.

Let ${\rm Rad}_p$ be the closed subspace spanned by the Rademacher sequence  $(\e_n)$ in $L_p([0,\,1])$. Then the noncommutative Khintchine inequality can be reformulated in terms of column and row spaces (see \cite{LPP, pis-NCLp}). To this end, we introduce
 $$CR_p=C_p+R_p\;\text{for}\; p\le 2\quad \text{and}\quad CR_p=C_p\cap R_p\;\text{for}\; p> 2.$$

\begin{lem}\label{kintchine}
 Let $1<p<\8$. Then ${\rm Rad}_p=CR_p$ completely isomorphically. Moreover, the orthogonal projection from $L_2([0,\,1])$ onto ${\rm Rad}_2$ extends to a completely bounded projection from $L_p([0,\,1])$ onto ${\rm Rad}_p$. All relevant constants depend only on $p$.
\end{lem}

If $E$ is an operator space, $C_p(E)$ (resp.
$R_p(E)$) denotes the first column (resp. row) subspace of the $E$-valued Schatten class $S_p(E)$. It
is clear that $C_p(E)$ and $R_p(E)$ are completely 1-complemented in $S_p(E)$. Consequently, applying \eqref{comVNLp} to $\M=B(\el_2)$, we get
 \beq\label{comVCp}
 \big(C_{p_0}(E_0),\, C_{p_1}(E_1)\big)_{\t}=C_{p}((E_0,\, E_1)_\t)
 \eeq
for any compatible pair $(E_0,\, E_1)$ of operator spaces, where $\frac1p=\frac{1-\t}{p_0}+\frac{\t}{p_1}$.

Note that $C_p(R_p)=S_p$ and $C_p(C_p)=S_2$ isometrically at the Banach space level. Here we represent the elements in $C_p(R_p)$ and $C_p(C_p)$ by infinite matrices in the canonical bases of $C_p$ and $R_p$. The following elementary fact is known to experts and implicitly contained in the proof of \cite[Lemma~5.9]{xu-emb}. We include its simple proof for completeness.

\begin{lem}\label{CR}
 Let $1\le p,q\le\8$. Let $r$ be determined by $\frac1r=\frac1{2p}+\frac1{2q'}$. Then
  $$C_p(CR_q)=S_r\;\text{ with equivalent norms}.$$
 \end{lem}

\begin{proof}
  By \eqref{CR},  \eqref{comVCp} and \eqref{duality Cp-Rp}, we have the following isometric equalities
 \begin{align*}
  C_\8(C_q)
  &=\big(C_{\8}(C_\8),\, C_{\8}(C_1)\big)_{\frac1q}\\
  &=\big(C_{\8}(C_\8),\, C_{\8}(R_\8)\big)_{\frac1q}\\
  &=\big(S_2,\, S_\8\big)_{\frac1q}=S_{2q'}\,.
  \end{align*}
Similarly, $C_1(C_q)=S_{(2q)'}$ isometrically. Thus
  $$C_p(C_q)
  =\big(C_{\8}(C_q),\, C_{1}(C_q)\big)_{\frac1p}
  =\big(S_{2q'},\,S_{(2q)'}\big)_{\frac1p}=S_r.$$
Combining this with \eqref{duality Cp-Rp}, we also have
 $$C_p(R_q)=C_p(C_{p'}) =S_t\;\text{ isometrically},$$
 where $\frac1t=\frac1{2p}+\frac1{2q}$. Hence,
 $$C_p(CR_q)=C_p(C_q)+C_p(R_q)=S_r+S_t=S_r$$
for $q\le2$ and $C_p(CR_q)= S_r\cap S_t=S_r$ for $q>2$ too.
\end{proof}

\begin{proof}[Proof of Theorem~\ref{Lpp}]
 By the type decomposition of von Neumann algebras, if $\M$ is not finite dimensional, then $\M$ contains an infinite dimensional commutative $L_\8(\O,\mu)$ as subalgebra which is moreover the image of a trace preserving normal conditional expectation. Indeed, if the type I summand of $\M$ is infinite dimensional, then $\M$ contains an infinite dimensional commutative $L_\8(\O,\mu)$. On the other hand, if the type II$_\8$ summand of $\M$ exists, then $\M$ contains $\B(\el_2)$, so $\el_\8$ too. Finally, if the type II$_1$ summand of $\M$ exists, then $\M$ contains $L_\8([0,\,1])$. See \cite{Takesaki} for the type decomposition of von Neumann algebras.

Note that if $L_\8(\O,\mu)$ is infinite dimensional,  $L_\8(\O,\mu)$ contains, as subalgebra, either $L_\8([0,\,1])$ or $\el_\8$. On the other hand, if $L_{p,p}(\M)=L_p(\M)$ held for $\M=\el_\8$, it would do so for $\M=\el^n_\8$ uniformly in $n\ge1$. Then by a standard approximation argument, we see that it would hold for $\M=L_\8([0,\,1])$ too.

 Thus we are reduced to showing the theorem for the special case $\M=L_\8([0,\,1])$. Namely, we must show that $L_{p,p}([0,\,1])\not=L_p([0,\,1])$ completely isomorphically. In the rest of the proof, we will drop the interval $[0,\,1]$ from $L_{p}([0,\,1])$. By \cite[Theorem~5.4]{xu-op}, we choose $1<p_0<p_1<\8$ and $0<\t<1$ such that $\frac1p=\frac{1-\t}{p_0}+ \frac{\t}{p_1}$ and
  $$L_{p,p}=\big(L_{p_0},\, L_{p_1}\big)_{\t, p}\,.$$
Then by \cite[Proposition~6.3]{xu-op} and its proof, we have
  $$S_p(L_{p,p})=\big(S_p(L_{p_0}),\, S_p(L_{p_1})\big)_{\t, p}\,.$$
Using the complete complementation of $C_p(E)$ in $S_p(E)$, we deduce
$$C_p(L_{p,p})=\big(C_p(L_{p_0}),\, C_p(L_{p_1})\big)_{\t, p}\,.$$
On the other hand, by the complete complementation of ${\rm Rad}_{p_i}$ in $L_{p_i}$ for $i=0,1$, we get the following isomorphic embedding:
 $$\big(C_p({\rm Rad}_{p_0}),\, C_p({\rm Rad}_{p_1})\big)_{\t, p}\subset C_p(L_{p,p}).$$
Now by Lemmas~\ref{kintchine} and \ref{CR},
 $$\big(C_p({\rm Rad}_{p_0}),\, C_p({\rm Rad}_{p_1})\big)_{\t, p}
 =(S_{r_0},\, S_{r_1})_{\t, p}=S_{2,p}\;\text{ with equivalent norms},$$
where $\frac1{r_i}=\frac1{2p}+\frac1{2p_i'}$ for $i=0, 1$. Thus the closed subspace spanned by the Rademacher functions  in $C_p(L_{p,p})$ is isomorphic to $S_{2,p}$. However, that spanned by the same functions in $C_p(L_{p})$ is isomorphic to $S_{2}$. But $S_{2,p}=S_2$ only if $p=2$. Thus the theorem is proved.
 \end{proof}


\section{Proof of Theorem~\ref{Lppb}}\label{Proof of Theorem2}


We begin this proof by recalling the definition of the space $L_p(\M;\el_q)$ that is introduced in \cite{ju-doob} for $q=1$ and $q=\8$  (see also \cite{ju-xu-erg}), and in \cite{ju-xu} for $1<q<\8$. This definition is inspired by Pisier's description of the norm of $L_p(\M;\el_q)$ in the hyperfinite case.

A sequence $(x_i)$ in $L_p(\M)$ belongs to $L_p(\M;\ell_\8)$ iff $(x_i)$ admits a factorization $x_i=ay_ib$ with $a, b\in L_{2p}(\M)$ and $(y_i)\in\ell_\8(L_\8(\M))$. The norm of $(x_i)$ is then defined as
 \begin{equation}\label{norm of Lplinfty}
  \|(x_i)\|_{L_p(\M;\ell_\8)}
 =\inf_{x_i=ay_ib}\, \|a\|_{L_{2p}(\M)}\,
 \|(y_i)\|_{\ell_\8(L_\8(\M))}\, \|b\|_{L_{2p}(\M)}\,.
 \end{equation}
On the other hand,  $L_p(\M;\ell_1)$ is defined as the space of all sequences $(x_i)\subset L_p(\M)$ for which there exist $a_{ij}, b_{ij}\in L_{2p}(\M)$ such that
 \[x_i=\sum_ja_{ij}^*b_{ij}\,.\]
$L_p(\M;\ell_1)$ is equipped with the norm
 \[ \|(x_i)\|_{L_p(\M;\ell_1)}
 =\inf_{x_i=\sum_{j} a_{ij}^*b_{ij}}\,
 \big\|\sum_{i,j} a_{ij}^*a_{ij}\big\|_{L_{p}(\M)}^{\frac12}\,
 \big\|\sum_{i,j} b_{ij}^*b_{ij}\big\|_{L_{p}(\M)}^{\frac12} \,.\]
Now for $1<q<\8$ we define $L_p(\M;\ell_q)$ as a
complex interpolation space between $L_p(\M;\ell_\8)$ and
$L_p(\M;\ell_1)$:
 \[ L_p(\M;\ell_q)=
 \big(L_p(\M;\ell_\8),\; L_p(\M;\ell_1)\big)_{\frac1q} \,.\]

 The following description of the norm of $L_p(\M;\ell_q)$ is proved in \cite{pis-NCLp} for hyperfinite $\M$ and in \cite{ju-xu} for a general $\M$.

 \begin{lem}\label{explicit norm of Lplq}
 Let $1\le p, q\le\8$.
 \begin{enumerate}[\rm(i)]
 \item If $p\le q$,
 \[\|(x_i)\|_{L_p(\M;\ell_q)}=\inf_{x_i=a y_ib}\,
 \|a\|_{L_{2r}(\M)}\,\|(y_i)\|_{\ell_q(L_q(\M))}\,
 \|b\|_{L_{2r}(\M)}\]
for any $(x_i)\in L_p(\M;\ell_q)$, where $\frac1r=\frac1p-\frac1q$.
 \item If $p\ge q$,
  \[\|(x_i)\|_{L_p(\M;\ell_q)}
  =\sup_{\|\a\|_{L_{2s}(\M)}\le1,\;\|\beta\|_{L_{2s}(\M)}\le1 }\,
 \|(\a x_i\beta)\|_{\ell_q(L_q(\M))}\]
for any $(x_i)\in L_p(\M;\ell_q)$, where $\frac1s=\frac1q-\frac1p$.
 \end{enumerate}
\end{lem}

We will again consider the column subspace $C_p(\el_q)$ of $S_p(\el_q)$ for the proof of Theorem~\ref{Lppb}. As in the previous section, a generic element $x\in C_p(\el_q)$ is viewed as an infinite matrix
 $$x=\sum_{i,j=1}^\8 x_{ij}e_{i, 1}\otimes e_j.$$
Let $D_{p,q}$ denote the diagonal subspace of $C_p(\el_q)$ consisting of all $x$ with $x_{ij}=0$ for $i\neq j$.

\begin{lem}\label{pq complemented}
 Let $1\le p, q\le\8$. Then $D_{p,q}$ is completely $1$-complemented in $C_p(\el_q)$.
 \end{lem}

\begin{proof}
 The proof is very simple. It suffices to note that the canonical bases of $C_p$ and $\el_q$ are completely $1$-unconditional. A standard average argument then yields the assertion.
\end{proof}

We will identify an element $x=(x_i e_{i,1})\in D_{p,q}$ with the sequence $(x_i)$.

\begin{lem}\label{pq}
 Let $1\le p, q\le\8$ and $r_{p, q}$ be determined by $\frac1{r_{p,q}}=\frac1{2p}+\frac1{2q}$. Then $D_{p,q}=\el_{r_{p,q}}$ with equal norms.
 \end{lem}

\begin{proof}
 First consider the cases $q=\8$ and $q=1$. Let $x=(x_i e_{i,1})\in D_{p,\8}\subset S_p(\el_\8)$. Let $a$ be the diagonal matrix with the $x_i$'s as its diagonal entries. Then we have the following factorization:
  $$x_i e_{i,1}=a e_{i,1} e_{11},\quad i\ge1.$$
 Thus by the definition of the norm of $S_p(\el_\8)$, we get
  $$\|x\|_{S_p(\el_\8)}\le \|a\|_{S_{2p}}\, \|(e_{i,1})\|_{\el_\8(B(\el_2))}\,\| e_{11}\|_{S_{2p}}=\|x\|_{\el_{2p}}\,.$$
On the other hand, for $q=1$ we factorize $x$ as
  $$x_i e_{i,1}=[{\rm sgn}(x_i)|x_i|^\a e_{i,1}]\,[|x_i|^{1-\a}e_{1,1}]=a_i b_i,$$
 where $\a=\frac{r_{p,1}}{2p}$. Therefore, by the definition of the norm of $S_p(\el_1)$
  \begin{align*}
  \|x\|_{S_p(\el_1)}
  &\le\big\|\sum_i a_ia_i^*\big\|_{p}^{\frac12}\big\|\sum_i b^*_ib_i\big\|_{p}^{\frac12}\\
  &=\big(\sum_i |x_i|^{r_{p,1}}\big)^{\frac1{2p}}\big(\sum_i |x_i|^{r_{p,1}}\big)^{\frac12}=\|x\|_{\el_{r_{p,1}}}\,.
  \end{align*}
 Thus we have proved that for any $1\le p\le\8$
  $$\el_{r_{p,\8}}\subset D_{p,\8}\quad\text{ and }\quad \el_{r_{p,1}}\subset D_{p,1}\;\text{ contractively}.$$
 Dualizing these inclusions and using Lemma~\ref{pq complemented}, we deduce the assertion for $q=\8$ and $q=1$. The case $1<q<\8$ is then completed by complex interpolation via \eqref{comVCp} with the help of Lemma~\ref{pq complemented} again.
\end{proof}

\begin{rk}
 The previous lemma can be proved directly by Lemma~\ref{explicit norm of Lplq} without passing to complex interpolation.
 \end{rk}

\begin{proof}[Proof of Theorem~\ref{Lppb}]
 If $\M$ is isomorphic, as Banach space, to some commutative $L_\8$, then $\M$ is a finite direct sum of algebras of the form $L_\8(\O,\mu)\otimes \mathbb M_n$, where $\mathbb M_n$ is the $n\times n$ full matrix algebra. Then \eqref{Lpinfinite} goes back to \eqref{realVLp}.

 Conversely, suppose that $\M$ is not isomorphic to a commutative von Neumann algebra. Our first step is to reduce the non validity of \eqref{Lpinfinite} to the special case  where $\M=B(\el_\8)$. To this end, we use the type decomposition of $\M$. If the type I summand of $\M$ is infinite dimensional, then $\M$ contains $\mathbb M_n$ for infinite many $n$'s. On the other hand, if the type II$_\8$ summand of $\M$ exists, then $\M$ contains $\B(\el_2)$. Finally, it is well known that if the type II$_1$ summand of $\M$ exists, then $\M$ contains the hyperfinite II$_1$ factor $\R$ (cf. e.g. \cite{M}); $\R$ is the von Neumann tensor of countable many copies of $(\mathbb M_2,\,{\rm tr})$, where ${\rm tr}$ is the normalized trace on $\mathbb M_2$; so $\M$ again contains $\mathbb M_n$ for infinite many $n$'s. Note that in all the three cases, the $\mathbb M_n$'s contained in $\M$ are images of  trace preserving normal conditional expectations (up to a normalization in the type II$_1$ case).

 In summary, if $\M$ is not isomorphic to a commutative von Neumann algebra, $\M$ contains $\mathbb M_n$ for infinite many $n$'s which are images of  trace preserving normal conditional expectations. This shows that if \eqref{Lpinfinite} held for $\M$, then it would do so for $\M=\mathbb M_n$ for infinite many $n$'s; consequently, by approximation, it would further hold for $\M=B(\el_2)$ too. This finishes the announced reduction.

\medskip

 It remains to show that  \eqref{Lpinfinite} fails for $\M=B(\el_2)$. Namely, we must show
  $$\big(S_\8(\el_q),\, S_1(\el_q)\big)_{\frac1p, p}\not=S_p(\el_q)\,.$$
 This is an easy consequence of the previous two lemmas. By Lemma~\ref{pq complemented},
 $$\big(D_{\8,q},\, D_{1,q}\big)_{\frac1p, p}\subset\big(S_\8(\el_q),\, S_1(\el_q)\big)_{\frac1p, p},\quad\text{an isometric embedding}.$$
On the other hand, by Lemma~\ref{pq},
 $$\big(D_{\8,q},\, D_{1,q}\big)_{\frac1p, p}=\big(\el_{r_{\8,q}},\, \el_{r_{\8,q}}\big)_{\frac1p, p}=\el_{r_{p,q}, p}\quad\text{ with equivalent norms},$$
where $\el_{r, p}$ denotes the Lorentz sequence space. On the other hand, by Lemma~\ref{pq},  the corresponding subspace of $S_p(\el_q)$ is equal to $\el_{r_{p,q}}$. However, $\el_{r_{p,q}, p}=\el_{r_{p,q}}$ if and only if $r_{p,q}=p$, i.e., $q=p$. The theorem is thus proved.
 \end{proof}


\section{Proof of Theorem~\ref{decreasing}}


In the sequel, $\|\,\|_p$ will denote the norm of $L_p(\M)$. Fix a finite sequence $(x_i)\subset L_p^+(\M)$. We claim that the function $q\mapsto
\big\|\big(\sum_ix_i^q\big)^{\frac1q}\big\|_p^q$ is log-convex.
Namely, for any $q_0,q_1\in(0,\,\8)$ and any $\a\in(0, \,1)$
 \beq\label{logconv}
 \big\|\big(\sum_ix_i^q\big)^{\frac1q}\big\|_p^q
 \le \big\|\big(\sum_ix_i^{q_0}\big)^{\frac1{q_0}}\big\|_p^{(1-\a)q_0}\,
 \big\|\big(\sum_ix_i^{q_1}\big)^{\frac1{q_1}}\big\|_p^{(1-\a)q_1}\,,
 \eeq
where $q=(1-\a)q_0 +\a q_1$. It suffices to show this inequality for $\a=\frac12$. Then it immediately follows from the H\"older inequality for column spaces:
 $$
 \big\|\sum_i x_i^q\big\|_{\frac{p}q}
 =\big\|\sum_i x_i^{\frac{q_0}2}\,x_i^{\frac{q_1}2}\big\|_{\frac{p}q}
 \le
 \big\|\big(\sum_ix_i^{q_0}\big)^{\frac1{q_0}}\big\|_p^{\frac{q_0}2}\,
 \big\|\big(\sum_ix_i^{q_1}\big)^{\frac1{q_1}}\big\|_p^{\frac{q_1}2}\,.
 $$

Now let us show  \eqref{decreasing in q}. Replacing $x_i$ by
$x_i^r$ and dividing all indices by $r$, we are reduced to the
case where $r=1<q$ and $p\ge1$. Thus it suffices to show
 \beq\label{decreasing in qbis}
 \big\|\big(\sum_i x_i^q\big)^{\frac1q}\big\|_p
 \le\big\|\sum_i x_i\big\|_p\,.
 \eeq
Set $x=\sum_ix_i$. We first consider the case where $q=2$. If
$p\le2$, then
 \begin{align*}
 \big\|\big(\sum_i x_i^2\big)^{\frac12}\big\|_p^p
 &=\tau\big[\big(\sum_i x_i^2\big)^{\frac{p}2}\big]\le \sum_i \tau(x_i^p)\\
 &\le\sum_i \tau(x_i^{\frac12}x^{p-1}x_i^{\frac12})
 =\sum_i \tau(x^{p-1}x_i)=\|x\|_p^p\,.
  \end{align*}
Note that the same argument yields \eqref{decreasing in qbis} in
the case where $p\le \min(q, 2)$. Assume $p>2$. Let $(\e_n)$ be a
Rademacher sequence, and let $\mathbb E$ denote the corresponding
expectation. Then by the triangle inequality in $L_{\frac{p}2}(\M)$,
 $$
 \big\|\big(\sum_i x_i^2\big)^{\frac12}\big\|_p
 \le\big[\mathbb E\big\|\sum_i \e_ix_i\big\|_p^2\big]^{\frac12}
 \le\|x\|_p\,,
$$
where we have used the fact that $-x\le \sum_i\e_ix_i\le x$. Thus
inequality \eqref{decreasing in qbis} is proved for $q=2$ and any
$p\ge1$.

Using \eqref{logconv} and the just proved cases, we deduce
\eqref{decreasing in qbis} for any $1<q\le2$ and any $p\ge1$.

Next, we show \eqref{decreasing in qbis} for $q=p$. If $p\le 2$,
this was proved previously. The case $p>2$ easily follows by an
iteration argument. Indeed, if $2<p\le 4$, by the two cases
already proved
  \begin{align*}
 \big\|\big(\sum_i x_i^p\big)^{\frac1p}\big\|_p
 &=\big\|\big(\sum_i x_i^p\big)^{\frac2p}\big\|_{\frac{p}2}^{\frac12}
 \le\big\|\sum_i x_i^2\big\|_{\frac{p}2}^{\frac12}\\
 &=\big\|\big(\sum_n x_n^2\big)^{\frac12}\big\|_{p}\le\|x\|_p.
  \end{align*}
Repeating this argument, we obtain \eqref{decreasing in qbis} for
the case $q=p$.

We then deduce \eqref{decreasing in qbis} for $1\le p\le q$ as follows
 $$
 \big\|\big(\sum_i x_i^q\big)^{\frac1q}\big\|_p
 \le\big\|\big(\sum_i x_i^p\big)^{\frac1p}\big\|_p
 \le\|x\|_p\,.$$

Thus it remains to consider the case where
$p\ge q\ge 2$. This is treated by an iteration argument as above. Indeed, if $q\le 4$, then
  $$
 \big\|\big(\sum_i x_i^q\big)^{\frac1q}\big\|_p
 =\big\|\big(\sum_i (x_i^2)^{\frac{q}2}\big)^{\frac2q}\big\|_{\frac{p}2}^{\frac12}
 \le\big\|\big(\sum_i x_i^2\big)^{\frac12}\big\|_p\le \|x\|_p\,.
 $$
Thus the proof of (\ref{decreasing in qbis}), so that of   \eqref{decreasing in q}, is complete.

\medskip

Now we turn to the proof of part (ii) of the theorem. As in the proof of Theorem~\ref{Lppb}, it suffices to consider the case where
$\M=B(\el_2)$. Suppose that \eqref{negative} holds with a constant $C$ and some indices $p, q, r$ such
that $0<p<r<q<\8$. Then by dividing all indices by $r$, we are
again reduced to the case where $p<1=r<q$. Thus
 \beq\label{nondecresing}
 \big\|\big(\sum_ix_i^q\big)^{\frac1q}\big\|_p
 \le C\,\big\|\sum_ix_i\big\|_p
 \eeq
for all finite sequences $(x_i)\subset L_p^+(\M)$.
We claim that $C$ must be equal to $1$. Indeed, given a positive
integer $k$, applying (\ref{nondecresing}) to the family
$(x_{i_1}\ot\cdots\ot x_{i_k})$ we get
  $$
 \big\|\big(\sum_{i_1, ..., i_k}x_{i_1}^q\ot\cdots\ot
 x_{i_k}^q\big)^{\frac1q}\big\|_p
 \le C\,\big\|\sum_{i_1, ..., i_k}x_{i_1}\ot\cdots\ot x_{i_k}\big\|_p\,.
 $$
However,
 $$
 \big\|\sum_{i_1, ..., i_k}x_{i_1}\ot\cdots\ot x_{i_k}\big\|_p
 =\big\|\big(\sum_ix_i\big)^{\ot k}\big\|_p
 =\big\|\sum_ix_i\big\|_p^k\,.
 $$
Similarly,
 $$\big\|\big(\sum_{i_1, ..., i_k}x_{i_1}^q\ot\cdots\ot
 x_{i_k}^q\big)^{\frac1q}\big\|_p
 =\big\|\big(\sum_ix_i^q\big)^{\frac1q}\big\|_p^k\,.$$
It then follows that
 $$
 \big\|\big(\sum_ix_i^q\big)^{\frac1q}\big\|_p
 \le C^{\frac1k}\,\big\|\sum_ix_i\big\|_p\,;
 $$
whence the claim by letting $k\to\8$.

Now it is easy to construct counterexamples to
(\ref{nondecresing}) with $C=1$. Consider $\M=\mathbb{M}_2$ and  the following matrices
 $$x=\begin{pmatrix}
  1 & 1 \\
  1 & 1 \\
\end{pmatrix}
 \quad\text{ and }\quad
 y=\begin{pmatrix}
  0 & 0 \\
  0 & t\\
\end{pmatrix}$$
with $t>0$ very small. Then
 $$x+y=\begin{pmatrix}
  1 & 1 \\
  1 & 1+t \\
\end{pmatrix} \quad\text {and }\quad
 x^q+y^q=2^{q-1}\begin{pmatrix}
  1 & 1 \\
  1 & 1+2^{1-q}\,t^q\\
\end{pmatrix}.$$
The two eigenvalues of $x+y$ are
 $$\frac{2+t\pm\sqrt{4+t^2}}2\,.$$
Thus (recalling that $p<1$)
 \begin{align*}
 \big\|x+y\big\|_p^p
 &=\frac{(2+t +\sqrt{4+t^2}\,)^p}{2^p}\, +\,
 \frac{(2+t -\sqrt{4+t^2}\,)^p}{2^p}\\
 &=2^p + 2^{-p}\,t^p +{\rm o}(t^p)\quad\mbox{as}\quad t\to0.
 \end{align*}
Similarly,
 $$
 \big\|(x^q+y^q)^{\frac1q}\big\|_p^p=2^p + 2^{-\frac{p}q}\,t^p +{\rm o}(t^p).
 $$
Hence, by (\ref{nondecresing}) with $C=1$, we get
 $$2^p + 2^{-\frac{p}q}\,t^p+{\rm o}(t^p)
 \le 2^p + 2^{-p}\,t^p +{\rm o}(t^p).$$
It then follows that $2^q\le 2$, which is a contradiction since $q>1$. Thus Theorem~\ref{decreasing} is completely proved.

\bigskip \n{\bf Acknowledgements.}  The second named author is partially supported by the French ANR project (No. ANR-19-CE40-0002) and the Natural Science Foundation of China (No.12031004).

\bigskip



\begin{thebibliography}{10}


\bibitem{BL}
 J.~Bergh,  and J.~L{\"o}fstr{\"o}m. {\em Interpolation spaces.}  Springer-Verlag, Berlin, 1976.

\bibitem{ER}
 Ed.~Effros, and  Z-J.~Ruan. {\em Operator spaces}, vol. 23 of {\em London Mathematical Society Monographs. New Series}. The Clarendon Press Oxford University Press, New York, 2000.

 \bibitem{ju-doob}
M.~Junge. Doob's inequality for non-commutative martingales. {\em J. Reine Angew. Math.}, 549 (2002),149--190.

\bibitem{ju-xu-erg}
 M.~Junge, and Q.~Xu. Noncommutative maximal ergodic inequalities. \textit{J. Amer. Math. Soc.} 20 (2007), 385-439.

\bibitem{ju-xu}
  M.~Junge, and Q.~Xu. Noncommutative Burkholder/Rosenthal inequalities II: Applications. {\em Israel J. Math.} 167 (2008), 227--282.

\bibitem{LPP}
F.~Lust-Piquard, and G.~Pisier. Noncommutative {K}hintchine and {P}aley inequalities. {\em Ark. Mat.}, 29 (1991), 241--260.

 \bibitem{M}
 J.L.~Marcolino Nhany. La stabilit\'e des espaces $L_p$ non-commutatifs. {\em Math. Scand.} 81 (1997), 212-218.

\bibitem{pis-OH}
 G.~Pisier. The operator {H}ilbert space {${\rm OH}$},
complex interpolation and tensor norms. {\em Mem. Amer. Math. Soc.}, 122(585), 1996.

\bibitem{pis-NCLp}
  G.~Pisier. Non-commutative vector valued {$L\sb p$}-spaces and completely
  {$p$}-summing maps. {\em Ast\'erisque}, vol. 247, 1998.

\bibitem{pis-intro}
 G.~Pisier.  {\em Introduction to operator space theory}, vol. 294 of {\em  London Mathematical Society Lecture Note Series}.  Cambridge University Press, Cambridge, 2003.

\bibitem{px}
G.~Pisier, and Q.~Xu. Non-commutative {$L\sp p$}-spaces. In {\em Handbook of the geometry of Banach spaces, Vol.\ 2}, pages
  1459--1517. North-Holland, Amsterdam, 2003.

\bibitem{Takesaki}
M.~Takesaki. {\em Theory of operator algebras } I. Springer-Verlag, 1979.
\bibitem{xu-op}
 Q.~Xu. Interpolation of operator spaces. {\em J. Funct. Anal.}, 139 (1996), 500--539.

 \bibitem{xu-emb}
Q.~Xu. Embedding of {$C\sb q$} and {$R\sb q$} into noncommutative {$L\sb
  p$}-spaces, {$1\leq p<q\leq2$}. {\em Math. Ann.}  335 (2006),109--131.



\end{thebibliography}
\end{document}